\documentclass[12pt]{amsart}
\usepackage{amssymb}
\usepackage{comment}
\usepackage[margin=1in]{geometry} 
\usepackage[bookmarks, bookmarksdepth=2, colorlinks=true, linkcolor=blue, citecolor=blue, urlcolor=blue]{hyperref}
\usepackage{xcolor}
\usepackage{eucal}
\usepackage{enumitem}

\newtheorem{theorem}{Theorem}[section]
\newtheorem{lemma}[theorem]{Lemma}
\newtheorem{proposition}[theorem]{Proposition}
\newtheorem{corollary}[theorem]{Corollary}
\theoremstyle{definition}
\newtheorem{definition}[theorem]{Definition}

\newtheorem{question}[theorem]{Question}
\newtheorem{remark}[theorem]{Remark}

\newcommand{\defn}[1]{\emph{#1}}
\newcommand{\arxiv}[1]{\href{http://arxiv.org/abs/#1}{{\tiny\tt arXiv:#1}}}
\newcommand{\DOI}[1]{\href{http://doi.org/#1}{\color{purple}{\tiny\tt DOI:#1}}}

\let\ol\overline
\let\ul\underline

\newcommand{\ep}{\varepsilon}
\newcommand{\res}{\textnormal{Res}}

\newcommand{\Z}{\mathbf{Z}}
\newcommand{\N}{\mathbf{N}}
\newcommand{\Q}{\mathbf{Q}}
\newcommand{\R}{\mathbf{R}}
\newcommand{\bP}{\mathbf{P}}
\newcommand{\A}{\mathbf{A}}

\newcommand{\cE}{\mathcal{E}}

\newcommand{\cP}{\mathcal{P}}
\newcommand{\cV}{\mathcal{V}}
\renewcommand{\sp}{\textnormal{span}}
\newcommand{\codim}{\textnormal{codim}}

\newcommand{\str}{\textnormal{str}}
\newcommand{\astr}{\textnormal{astr}}
\newcommand{\Gr}{\mathbf{Gr}}

\title{Two improvements in Birch's theorem on forms}
\date{\today}

\author{Amichai Lampert}
\address{Department of Mathematics, University of Michigan, Ann Arbor, MI}
\email{\href{mailto:amichai@umich.edu}{amichai@umich.edu}}
\thanks{AL was supported by NSF grants DMS-2402041 and DMS-1926686.}

\author{Andrew Snowden}
\address{Department of Mathematics, University of Michigan, Ann Arbor, MI}
\email{\href{mailto:asnowden@umich.edu}{asnowden@umich.edu}}
\urladdr{\url{http://www-personal.umich.edu/~asnowden/}}
\thanks{AS was supported by NSF grant DMS-2301871.}

\begin{document}

\begin{abstract}
Let $K$ be a \defn{Birch field}, that is, a field for which every diagonal form of odd degree in sufficiently many variables admits a non-zero solution; for example, $K$ could be the field of rational numbers. Let $f_1, \ldots, f_r$ be homogeneous forms  of odd degree over $K$ in $n$ variables, and let $Z$ be the variety they cut out. Birch proved if $n$ is sufficiently large then $Z(K)$ contains a non-zero point. We prove two results which show that $Z(K)$ is actually quite large. First, the Zariski closure of $Z(K)$ has bounded codimension in $\A^n$. And second, if the $f_i$'s have sufficiently high strength then $Z(K)$ is in fact Zariski dense in $Z$. The proofs use recent results on strength, and our methods build on recent work of Bik, Draisma, and Snowden, which established similar improvements to Brauer's theorem on forms.
\end{abstract}

\maketitle
\tableofcontents

\section{Introduction}

The arithmetic of forms in many variables is a subject that has received much attention over the years; see \cite{Greenberg} for a general introduction. Brauer \cite{Brauer} proved an important result in this area in 1945: he showed that, over certain types of fields (including all totally imaginary number fields), any system of forms in sufficiently many variables admits a non-trivial solution. In the predecessor to this paper \cite{BDS}, we established some strengthenings of Brauer's theorem. In particular, over the relevant fields---which we named \defn{Brauer fields}---we showed that the Zariski closure of the set of solutions has bounded codimension. So, not only is the set of non-trivial solutions non-empty, but it is really quite large.

The rational numbers are obviously not a Brauer field, since a positive definite quadratic form will never have a non-trivial zero, no matter how many variables there are. About twelve years after Brauer proved his theorem, Birch \cite{Birch} proved a variant that circumvents this issue: he showed that over certain fields, which we name \defn{Birch fields}, any system of odd degree forms in sufficiently many variables admits a non-trivial solution. The rational numbers are a Birch field, and, in fact, so is any number field. The purpose of this paper is to establish improvements to Birch's theorem analogous to the results of \cite{BDS}.

\subsection{Statement of results} \label{ss:results}

We now go about precisely stating our results. To begin, we introduce the relevant classes of fields.

\begin{definition} \label{defn:birch}
Let $K$ be a field. For a positive integer $d$, define $N_K(d)$ to be the minimal positive integer $n$ such that the following condition holds: for any non-zero $a_1,\ldots,a_n\in K$, the equation
\begin{displaymath}
a_1x_1^d + \ldots + a_nx_n^d=0
\end{displaymath}
admits a non-zero solution in $K$. We put $N_K(d)=\infty$ if no such $n$ exists. We say that $K$ is a \defn{Brauer field} (resp.\ \defn{Birch field}) if $N_K(d)$ is finite for all $d$ (resp.\ all odd $d$).
\end{definition}

Number fields are Birch fields by a theorem of Peck \cite[Theorem~3]{Peck}. In particular, the rational numbers $\Q$ are a Birch field\footnote{Birch states this was well-known before Peck's work, but as far as we can tell, it was just a folk theorem.}. We show that any finitely generated extension of $\R$ is a Birch field (Corollary~\ref{cor:fg-birch}). Of course, any Brauer field is also a Birch field, and many examples of Brauer fields are given in \cite{BDS}.

We assume that $K$ is a Birch field for the remainder of \S \ref{ss:results}. Let $\ul{f} = (f_1,\ldots,f_r)$ be a collection of forms (i.e., homogeneous polynomials) of odd degrees $\ul{d}=(d_1, \ldots, d_r)$ on a finite dimensional $K$-vector space $V$, and let $Z \subset V$ be the variety defined by the equations $f_i=0$. Birch's result alluded to above is the following theorem:

\begin{theorem}[Birch] \label{Birch}
If $\dim{V} \ge C_{\ref*{Birch}}(\ul{d})$ then $Z(K)\neq \{0\}.$  
\end{theorem}

Throughout this paper, when a new constant appears in a statement, we are implicitly asserting its existence. Thus the above theorem really means that a quantity $C_{\ref*{Birch}}$ exists such that the statement is true. All constants will (possibly) depend on the field $K$, but we omit this from the notation. Constants are labeled by the number of the statement in which they first appear.

Our main results improve on Birch's theorem by showing that $Z(K)$ is actually quite large. \emph{For the rest of the paper, we assume that $K$ has characteristic zero;} see Remark~\ref{rmk:char} for more details about the characteristic. Our first main result is:

\begin{theorem}\label{codim-gen}
The Zariski closure of $Z(K)$ has codimension at most $C_{\ref*{codim-gen}}(\ul{d})$.
\end{theorem}

To continue, we require the idea of strength of a polynomial, or a collection of polynomials (see \S \ref{s:bg}). This notion was first introduced by Schmidt \cite{Schmidt} (and is therefore sometimes called Schmidt rank), and was subsequently rediscovered by Green--Tao \cite{GT} and Ananyan--Hochster \cite{AH}. Schmidt's results in fact imply Theorem~\ref{codim-gen} for $K=\Q,$ by completely different methods. One of the main themes of the work of Ananyan--Hocshter is that if $\ul{f}$ has high collective strength then $f_1, \ldots, f_r$ behave like independent variables. Our second result is another instance of this theme:

\begin{theorem}\label{main}
The set $Z(K)$ is Zariski dense in $Z$, provided that the collective strength of $\ul{f}$ is at least $C_{\ref*{main}}(\ul{d})$.
\end{theorem}

We note that $\ul{f}$ has high collective strength if and only if the singular locus of $Z$ has high codimension (Proposition~\ref{prop:rk-sing}), so one could also phrase Theorem~\ref{main} using this hypothesis instead. We also note that the theorem ensures that there is a smooth $K$-point on $Z$; the existence of such a point is used in many applications in analytic number theory, see, for example, \cite{Birch2} and \cite{Skinner}.

We observe one other simple corollary of this theorem here.

\begin{corollary} \label{cor:main}
Let $d \ge 1$ be an odd integer. Then any equation
\begin{displaymath}
a_1 x_1^d + \cdots + a_n x_n^d = 1
\end{displaymath}
with $a_1, \ldots, a_n \in K$ non-zero and $n \ge C_{\ref*{cor:main}}(d)$ has a solution in $K$.
\end{corollary}

\begin{proof}
Let $f_1(x_1, \ldots, x_{n+1})=\sum_{i=1}^n a_i x_i^d - x_{n+1}^d$, which has strength at least $(n+1)/2$ (Proposition~\ref{prop:rk-sing}). Thus if $(n+1)/2>C_{\ref*{main}}(d)$ then $Z(K)$ is dense in $Z$. We can therefore take a $K$-point with $x_{n+1} \ne 0$, and then scale to obtain a solution with $x_{n+1}=1$.
\end{proof}

\begin{remark} \label{rmk:char}
Our main theorems continue to hold in positive characteristic, provided $K$ is infinite and the characteristic is greater than each $d_i$. We have chosen to simply work in characteristic zero throughout this paper to keep the exposition cleaner. The results of Brauer, Birch, and \cite{BDS} all hold in arbitary characteristic. We do not know if our theorems continue to hold in low characteristic.
\end{remark}

\begin{remark}
We say that $K$ is a \defn{Leep--Starr} field if there exists an integer $s \ge 1$ such that $N_K(d)$ is finite for all $d$ not divisible by $s$. Leep and Starr \cite{LS} generalized Theorem~\ref{Birch} to this context. It seems likely that our results (and proofs) generalize to this setting, but we leave this avenue unexplored at present.
\end{remark}

\subsection{Comments on the proofs}

The general approach in this paper is similar to that of \cite{BDS}, however, there is one fundamental problem we must solve. To explain this, let us recall the basic idea in the proof of Brauer's theorem. We sketch the proof of the theorem for a single form of degree $d$, assuming the theorem is already known for arbitrary collections of forms of lesser degree. Thus let $f$ be a degree $d$ form on a vector space $V$. We must find a non-zero vector $v \in V$ such that $f(v)=0$, assuming $\dim(V)$ is sufficiently large.

We say that vectors $v_1, \ldots, v_n$ are \defn{$f$-orthogonal} if
\begin{equation} \label{eq:intro}
f(x_1 v_1+\cdots+x_n v_n) = f(v_1) x_1^d + \cdots + f(v_n) x_n^d
\end{equation}
holds for all $x_1, \ldots, x_n \in K$. We claim that we can find a linearly independent $f$-orthogonal sequence $v_1, \ldots, v_n$, provided $\dim(V) \gg n$. Suppose that we have found $v_1, \ldots, v_{n-1}$. We now look for $v_n$ in a complementary space to the span of $v_1, \ldots, v_{n-1}$. The condition that $v_1, \ldots, v_n$ is $f$-orthogonal amounts to some polynomial equations on $v_n$, all of which have degree $<d$. Thus, by our inductive hypothesis, we can find a non-zero solution to these equations provided that $\dim(V)$ is sufficiently large. This proves the claim.

Now simply take $n \ge N_K(d)$, and choose $x_1, \ldots, x_n \in K$, not all zero, so that \eqref{eq:intro} vanishes; this is possible by the definition of $N_K(d)$. Then $v=x_1 v_1+\cdots+x_n v_n$ is a non-zero vector satisfying $f(v)=0$, as required.

The main results of \cite{BDS} are proved in a similar manner, though the arguments are more intricate. The basic idea is to find an $f$-orthogonal sequence $v_1, \ldots, v_n$ for which each $f(v_i)$ is non-zero, together with some additional conditions, which essentially allows us to replace $f$ with a diagonal form of high strength. We then use a specialized argument to handle that case.

We can now describe the fundamental problem we face. Suppose that $f$ has odd degree $d$. If $v_1, \ldots, v_{n-1}$ is an $f$-orthogonal sequence, then the equations expressing $f$-orthogonality of $v_1, \ldots, v_n$ involve all degrees $<d$, including even degrees. Over a Birch field, we cannot solve even degree equations (in general), so we cannot necessarily extend our orthogonal sequence.

The solution to this problem, which goes back to Birch's proof, is to find the sequence $v_1, \ldots, v_n$ all at once, i.e., not inductively. The equations expressing $f$-orthogonality of this sequence are multi-homogeneous of total degree $d$, and in each equation some $v_i$ appears with odd degree $<d$. This, it turns out, enables the use of an inductive argument. So, in essence, this paper combines the ideas of \cite{BDS} with the ``all at once'' method of Birch.

\subsection{Outline} 

In \S \ref{s:bg}, we review material about strength of polynomials and show how to deduce Theorem~\ref{codim-gen} from Theorem~\ref{main}. In \S \ref{s:multihom}, we prove a multi-homogeneous version of Theorem~\ref{main}, assuming an inductive hypothesis. This is used in \S \ref{s:diag} to reduce Theorem~\ref{main} to a statement about diagonal forms, which is then proved in \S \ref{s:proof-end}. Finally, in \S \ref{s:birch}, we give some examples of Birch fields.

%\subsection*{Acknowledgments}

\section{Preliminaries on strength} \label{s:bg}

In this section, we collect various existing results regarding strength which will be needed in the course of the paper. Throughout, $K$ denotes an arbitrary field of characteristic~0.

\subsection{Definitions}

Let $V$ be a finite dimensional $K$-vector space. We write $\cP(V)$ for the space of polynomials on $V$, and $\cP_d(V)$ for the subspace consisting of those polynomials that are homogeneous of degree $d$. If we fix a basis $e_1, \ldots, e_n$ of $V$ then we have $\cP(V)=K[x_1, \ldots, x_n]$. 

We now define the notion of strength in increasing order of generality:
\begin{enumerate}
\item The \defn{strength} of a form $f\in \cP_d(V)$ is the minimal $s$ for which there exists an expression
\begin{displaymath}
f = \sum_{i=1}^s g_ih_i,
\end{displaymath}
where $g_i,h_i\in \cP(V)$ are homogeneous forms of degree $<d$. By convention, non-zero linear forms have infinite strength. The zero form (of any degree) has strength zero.
\item The \defn{strength} of a collection of forms $f_1,\ldots,f_r\in \cP_d(V)$ of the same degree is the minimal strength of a non-trivial linear combination. Note that if the $f_i$'s are linearly dependent then their collective strength is zero.
\item The \defn{strength} of a collection of forms of varying degrees is the minimum strength of the forms in each degree separately.

\end{enumerate}
We write $\str(-)$ to denote any of these notions of strength. In the following results, $\ul{f} = (f_1,\ldots,f_r)$ denotes a sequence of forms degrees $\ul{d} = (d_1,\ldots,d_r)$ on $V$ and $s$ denotes a positive integer.

\subsection{Regularization and Theorem \ref{codim-gen}} \label{ss:reg}

Given two tuples $\ul{d}$ and $\ul{e}$ of positive integers, we define $\ul{d}<\ul{e}$ if $\ul{d}$ is less than $\ul{e}$ lexicographically, after first sorting the tuples to be in order. Concretely, to make $\ul{d}$ smaller, one iteratively replaces one of its entries with a list of strictly smaller numbers. For example, $(3,3,1)<(5,3)$. This is a well-order on (unordered) tuples.

The following regularization procedure allows one to put arbitrary forms in an ideal generated by high strength and is the key to deducing Theorem~\ref{codim-gen} from Theorem~\ref{main}.

\begin{proposition}[Regularization] \label{prop:reg}
Let $\Phi:\bigcup_{r\ge 1} \N^r\to \N$ be any function and suppose the degrees $d_i$ are all odd. Then there exist homogeneous forms $\ul{g}=(g_1, \ldots, g_R)$ on $V$ of odd degrees $\ul{e}=(e_1, \ldots, e_R)$ such that:
\begin{enumerate}
\item $R \le C_{\ref*{prop:reg}}(\ul{d},\Phi)$, and $\ul{e} \le \ul{d}$ in the order defined above,
\item $\str(\ul{g}) > \Phi(\ul{e})$ and
\item each $f_i$ belongs to the ideal $(g_1, \ldots, g_R)$.
\end{enumerate}
\end{proposition}

\begin{proof}
This was first proved in \cite[\S 2]{Schmidt}. See also \cite[Proposition~8.1]{ESS}.
\end{proof}

\begin{proof}[Proof of Theorem~\ref{codim-gen}, assuming Theorem \ref{main}]
Apply Proposition~\ref{prop:reg} with $\Phi(\ul{e})=C_{\ref*{main}}(\ul{e})$ and let $Y$ be the variety defined by the vanishing of the resulting forms $\ul{g}.$ Note that $Y\subset Z.$ By Theorem \ref{main}, $Y(K)$ is Zariski dense in $Y.$ This implies that the Zariski closure of $Z(K)$ contains $Y,$ so it has codimension at most $R\le C_{\ref*{prop:reg}}(\ul{d},\Phi).$ 
\end{proof}

\subsection{Field extensions}

The strength of a form can decrease upon passing to a field extension. For example, $x^2+y^2$ has strength~2 over the real numbers, but strength~1 over the complex numbers. We define the \defn{absolute strength} of a form (or collection of forms), denoted $\astr(-)$, to be the strength over the algebraic closure of $K$. We will require the following result, which essentially says that strength does not drop too much when passing to field extensions.

\begin{proposition} \label{rk-bar}
If $\str(\ul{f}) \ge C_{\ref*{rk-bar}}(\ul{d},s)$ then $\astr(\ul{f}) \ge s$.
\end{proposition}

\begin{proof}
This is proved in \cite{BDLZ}. (We note that \cite{BDLZ} also applies in positive characteristic greater than the $d_i$'s, while \cite{BDS2} proves a similar result in arbitrary characteristic.)
\end{proof}

\subsection{Geometric properties of strength}

The next result relates strength to singularities.

\begin{proposition} \label{prop:rk-sing}
Let $X \subset V$ be the locus where $\nabla f_1(x), \ldots, \nabla f_r(x)$ are linearly dependent. If $\str(\ul{f}) \ge C_{\ref*{prop:rk-sing}}(\ul{d},s)$ then the codimension of $X$ is at least $s$. On the other hand, the codimension is always at most $2\cdot \str(\ul{f}).$ 
\end{proposition}

\begin{proof}
When $K$ is algebraically closed, the lower bound is \cite[Theorem A]{AH}. We can deduce the general case using Proposition~\ref{rk-bar}. For the upper bound, suppose $0\neq f\in\sp(\ul{f}) $ is homogeneous and $f = \sum_{i=1}^t g_i h_i.$ Then 
\begin{displaymath}
V(g_1,\ldots,g_t, h_1,\ldots,h_t) \subset (\nabla f = 0)
\end{displaymath}
so by Krull's principal ideal theorem $\codim_{\A^n} (\nabla f = 0) \le 2t.$
\end{proof}

\begin{remark}
For quantitative results, see \cite{KLP}.
\end{remark}

\begin{corollary} \label{prime}
If $\str(\ul{f}) \ge C_{\ref*{prime}}(\ul{d})$ then $f_1,\ldots,f_r$ is a prime sequence, i.e., $f_1, \ldots, f_i$ generate a prime ideal for each $1 \le i \le r$.
\end{corollary}

Here is a related result which will also be useful.

\begin{proposition} \label{prop:surjective}
If $\str(\ul{f})\ge C_{\ref*{prop:surjective}} (\ul{d})$ then $\ul{f}:V\to \A^r$ is surjective (as a map of varieties) with geometrically irreducible fibers.
\end{proposition}

\begin{proof}
For $K$ which is algebraically closed, this is \cite[Theorem 1.11]{KZ-applications}. For the general case, combine that result with Proposition~\ref{rk-bar}.
\end{proof}

\subsection{Generic forms have high strength}

The following result makes this precise.

\begin{proposition} \label{prop:str-gen}
If $\dim{V} \ge C_{\ref*{prop:str-gen}}(\ul{d}, s)$ then there is a non-empty Zariski open subset $U$ of $\prod_{i=1}^r \cP_{d_i}(V)$ such that any $K$-point $\ul{f}$ of $U$ satisfies $\str(\ul{f})>s$.
\end{proposition}

\begin{proof}
We may assume that the degrees are all the same, i.e., $d_i=d$ for all $i$. Letting $n=\dim{V}$, we have
\begin{displaymath}
\dim(\cP_d(V)^r) = r\cdot\binom{d+n-1}{d} = \frac{r}{d!} \cdot n^d + O(n^{d-1}),
\end{displaymath}
where our implied constants are allowed to depend on all parameters besides $n.$
Let $X_s$ be the strength $\le s$ locus in $\cP_d(V)^r$. To complete the proof, we show that $X_s$ is a constructible set with
\begin{equation} \label{eq:Xdim}
\dim(X_s) = \frac{r-1}{d!} \cdot n^d + O(n^{d-1}).
\end{equation}
Thus, for $n$ sufficiently large, the complement of $X_s$ contains a non-empty open set.

Let $Y_s \subset \cP_d(V)$ be the strength $\le s$ locus.  For a tuple $e = (e_1,\ldots,e_s)$ with $1\le e_i\le d/2$ let
\begin{displaymath}
\phi_e \colon \prod_{i=1}^s \big( \cP_{d-e_i}(V) \times \cP_{e_i}(V) \big) \to \cP_d(V)
\end{displaymath}
be the map given by
\begin{displaymath}
\phi_e(g, h) = \sum_{i=1}^s g_i\cdot h_i.
\end{displaymath}
The dimension of the image of $\phi_e$ is bounded above by
\begin{align*}
\sum_{i=1}^s (\dim(\cP_{d-e_i}(V))+\dim(\cP_{e_i})(V)) =  \sum_{i=1}^s \left(\binom{n+d-e_i-1}{d-e_i} + \binom{n+e_i-1}{e_i}\right) = O(n^{d-1}).
\end{align*}
Since $Y_s$ is the finite union of the images of the $\phi_e$'s, over all choices of $e$, it is a constructible set whose dimension is also $O(n^{d-1})$.

Now, consider the map
\begin{align*}
\A^{r-1} \times \cP_d(V)^{r-1} \times Y_s & \to \cP_d(V)^r \\
((a_1, \ldots, a_{r-1}), (f_1, \ldots, f_{r-1}), g) &\mapsto (f_1, \ldots, f_{r-1}, a_1 f_1+\cdots+a_{r-1} f_{r-1} + g).
\end{align*}
There are $r-1$ other such maps, where the linear combination is put in a different coordinate. The joint images of these maps is $X_s$. We thus see that $X_s$ is constructible and \eqref{eq:Xdim} holds, as required.
\end{proof}

\subsection{The restriction map} \label{ss:res}

%In \S \ref{ss:res}, we assume that $K$ is infinite and its characteristic is either~0 or larger than all relevant degrees. 
It will be useful to phrase the next results using a more abstract definition of strength.

\begin{definition}
    Suppose $W$ is another finite dimensional $K$-vector space and $f \colon V \to W$ is a homogeneous polynomial map of degree $d$. Choose a linear isomorphism $W=K^r$, and let $f_1, \ldots, f_r \in \cP_d(V)$ be the components of $f$. The \defn{strength} of the map $f$ is the collective strength of $f_1, \ldots, f_r$. This is independent of the choice of basis of $W$.

    More generally, suppose that $W=W_1 \times \cdots \times W_d$ is a graded vector space and  $f \colon V \to W$ is a homogeneous algebraic map, meaning its $i$-th component is homogeneous of degree $i$. We define the \defn{strength} of $f$ to be the minimum of the strengths of its homogeneous components.
\end{definition}

Let $\ul{f}=(f_1, \ldots, f_r)$ be homogeneous forms of degrees $\ul{d}=(d_1, \ldots, d_r)$ on a vector space $W$. Given vectors $w_1, \ldots, w_{\ell} \in W$, there is an associated linear map $K^{\ell} \to W$, and we can consider the restriction of $\ul{f}$ to $K^{\ell}$. This defines a homogeneous algebraic map
\begin{displaymath}
\res_{\ul{f}} \colon W^{\ell} \to \prod_{i=1}^r \cP_{d_i}(K^{\ell}),
\end{displaymath}
where the $i$th factor on the right is placed in degree $d_i$. A key property of $\res_{\ul{f}}$ is the following, which is inspired by \cite{KZ-properties}.

\begin{proposition} \label{prop:res-strength0}
We have $\str(\res_{\ul{f}}) \ge \str(\ul{f})$.
\end{proposition}

In particular, if $\str(\ul{f})$ is sufficiently large then $\res_{\ul{f}}$ is a surjective map of varieties by Proposition~\ref{prop:surjective}. Thus, if $K$ is algebraically closed, then $\res_{\ul{f}}$ is surjective on $K$-points, meaning that any tuple of polynomials on $K^{\ell}$ of degrees $\ul{d}$ can be obtained as a specialization of $\ul{f}$. This is essentially\footnote{The result of \cite{KZ-properties} states that the restriction map is surjective when restricted to linearly independent tuples in $W^{\ell}$. One can obtain this stronger statement from what we have proved with a little more work.} the universality result of \cite{KZ-properties}.

We will need a slightly more general result. Let $\ul{f}=(f_1, \ldots, f_r)$ be bihomogeneous forms on $V \times W$, and let $(d_i, e_i)$ be the bidegree of $f_i$. For $v\in V,\ w_1, \ldots, w_{\ell} \in W,$ we obtain a linear map $K^{\ell}\to V\times W$ given by $(x_1,\ldots,x_{\ell})\mapsto (v,\sum_{i=1}^{\ell} x_iw_i).$ This induces a restriction map
\begin{displaymath}
\res_{\ul{f}} \colon V \times W^{\ell} \to \prod_{i=1}^r \cP_{e_i}(K^{\ell})
\end{displaymath}
The following proposition applied with $V=0$ recovers Proposition~\ref{prop:res-strength0}.

\begin{proposition} \label{prop:res-strength}
We have $\str(\res_{\ul{f}}) \ge \str(\ul{f})$.
\end{proposition}

\begin{proof}
This is a statement about each degree separately, so we may assume without loss of generality that the $f_i$ all have the same degree. Begin by expanding out each of the components 
\[
\res_{\ul{f}}(i)(v,w_1,\ldots,w_\ell) = \sum_{\ep\in(\Z^{\ge 0})^\ell,|\ep| = e_i} x^\ep f_i^\ep(v,w_1,\ldots,w_\ell), 
\]
where $f_i^\ep$ is multi-homogeneous of multi-degree $(d_i,\ep).$
Let $s = \str(\ul{f})$ and suppose there is a linear combination with
    \[
   \str\left( \sum_{i\in[r],|\ep| = e_i} c_{i,\ep} f_i^\ep(v,w_1,\ldots,w_\ell) \right) < s.
    \]
    Making the linear substitution $w_i = \lambda_i w$ for scalars $\lambda_1,\ldots,\lambda_\ell\in K$ leaves us with 
    \[
     \str\left( \sum_{i\in[r],|\ep| = e_i} c_{i,\ep} \binom{e_i}{\ep_1,\ldots,\ep_\ell}\lambda_1^{\ep_1}\ldots\lambda_\ell^{\ep_\ell} f_i(v,w) \right) < s.
    \]
    The definition of $s$ yields $\sum_{|\ep| = e_i} c_{i,\ep} \binom{e_i}{\ep_1,\ldots,\ep_\ell}\lambda_1^{\ep_1}\ldots\lambda_\ell^{\ep_\ell} = 0$ for all $i\in[r].$ Since this is true for any choice of  $\lambda_1,\ldots,\lambda_\ell\in K,$ we deduce $ c_{i,\ep} = 0 $ for all $i,\ep,$ which proves the claim. 
\end{proof}

\section{Multi-homogeneous forms} \label{s:multihom}

We now begin in earnest the proof of Theorem~\ref{main}. The main result of \S \ref{s:multihom} states that certain systems of multi-homogeneous equations admit a solution. This is the key tool that allows us to carry out Birch's ``all at once'' approach to finding orthogonal vectors in \S \ref{s:diag}, which we use to reduce the main theorem to the case of diagonal forms. Throughout, $K$ denotes a Birch field of characteristic~0 and $Z$ denotes the variety defined by the vanishing of the $f_i$'s.

\subsection{Set-up}

We say that a tuple $\ul{d}=(d_1, \ldots, d_r)$ of positive integers is \defn{odd} if all the $d_i$ are. For odd $\ul{d},$ consider the following statement:
\begin{description}[align=right,labelwidth=1.5cm,leftmargin=!]
\item[$\Sigma(\ul{d})$] There exists a quantity $C(\ul{d})$ such that if $\ul{f} = (f_1, \ldots, f_r)$ are forms of degrees $d_1, \ldots, d_r$ on a finite dimensional $K$-vector space $V$ and of strength at least $C(\ul{d})$ then $Z(K)$ is Zariski dense in $Z$.
\end{description}
To prove Theorem~\ref{main}, we must show that $\Sigma(\ul{d})$ holds for all odd $\ul{d}$. This will be achieved by induction with respect to the order in \S \ref{ss:reg}. Our inductive hypothesis is:
\begin{description}[align=right,labelwidth=1.5cm,leftmargin=!]
\item[$\Sigma^*(\ul{d})$] The statement $\Sigma(\ul{e})$ holds for all odd $\ul{e}<\ul{d}$.
\end{description}
To prove $\Sigma(\ul{d})$ holds for all $\ul{d}$, it suffices, by induction on $\ul{d}$, to prove that $\Sigma^*(\ul{d})$ implies $\Sigma(\ul{d})$. This is what we will eventually do. For notational convenience, we introduce one additional statement, depending on a single integer $d$:
\begin{description}[align=right,labelwidth=1.5cm,leftmargin=!]
\item[$\Sigma^{\dag}(d)$] The statement $\Sigma(\ul{e})$ holds for all odd $\ul{e}$ with $\max(\ul{e})<d$.
\end{description}
We will refer to the above three statements throughout \S \ref{s:multihom}--\S \ref{s:proof-end}.

\subsection{Main result}

We now turn towards out main result:

\begin{proposition} \label{prop:multi}
Suppose $\Sigma^{\dag}(d)$ holds. Let $\ul{f} = (f_{i,j})_{i\in [r],j\in [n_i]}$ be multi-homogeneous forms on $V_1 \times \cdots \times V_r$ such that $f_{i,j}$ has total degree $\le D$ and odd degree $<d$ on $V_i$. Then, assuming $\str(\ul{f}) \ge C_{\ref*{prop:multi}}(D,d,\sum_{i=1}^r n_i),$ the $K$-points of $Z$ are Zariski dense.
\end{proposition}

To prove this, we will need two geometric lemmas regarding collections of bihomogeneous forms. For a graded vector space $V=V_1 \times \cdots \times V_n$ and a multi-degree $d=(d_1, \ldots, d_n)$, we write $\cP_d(V)$ for the space of forms on $V$ that are multi-homogeneous of multi-degree $d$. Let $\ul{f} = (f_1,\ldots,f_r),\ \ul{g} = (g_1,\ldots,g_s) $ be bihomogeneous forms on $V\times W.$ Suppose that $\ul{f},\ul{g}$ have bidegrees $\ul{d} = ((d_i,d_{i}'))_{i\in[r]}$ and $\ul{e} = ((e_i,e_{i}'))_{i\in[s]},$ respectively.

\begin{lemma}\label{res-irr}
      Consider the following subvariety of $V\times W^\ell\times \A^\ell:$
     \[
    X = \{ (v,\mathbf{w},\mathbf{x}) : \res_{\ul{f}} (v,\mathbf{w}) \equiv 0,\ \ul{g}(v,\mathbf{x}\cdot \mathbf{w}) = 0 \}.
     \]
    If $\ell > s+3$ and $\str(\ul{f},\ul{g}) \ge C_{\ref*{res-irr}}(\ul{d},\ul{e},l)$ then $X$ is irreducible.
\end{lemma}

\begin{proof}
    $X$ is the zero locus of a map $\Phi: V\times W^\ell\times \A^\ell \to \prod_{i=1}^r \cP_{d_{i}'}(K^\ell)\times \A^s. $ Let $Y\subset X$ be the locus of points in $X$ where the differential of $\Phi$ is not surjective. By \cite[Corollaire XI 3.14]{SGA}, it's enough to show that $\codim_X Y > 3$ to deduce that $X = (\Phi = 0)$ is an irreducible complete intersection. Let $\pi: X \to \A^\ell$ be the projection and let  $X' = ((v,\mathbf{w},\mathbf{x}): \res_{\ul{f}} (v,\mathbf{w}) \equiv 0)$. Note that $\codim_{X'} \pi^{-1}(0) = l, $ so we have  $\codim_X \pi^{-1}(0)\cap Y \ge l-s > 3.$ For any fixed $\mathbf{x} \neq 0,$ we obtain a map $\Phi_{\mathbf{x}}(v,\mathbf{w}) := \Phi (v,\mathbf{w},\mathbf{x}).$ Then we have
    \[
    Y\cap \pi^{-1}(\mathbf{x}) \subset ((v,\mathbf{w},\mathbf{x})\in X: \textnormal{Differential of }\Phi_{\mathbf{x}}(v,\mathbf{w}) \textnormal{ is not surjective}).
    \]
    The components of $\Phi_{\mathbf{x}}$ are a linearly independent subset of the components of $\res_{\ul{f},\ul{g}},$ so by Proposition \ref{prop:res-strength} we have $\str(\Phi_{\mathbf{x}}) \ge C_{\ref*{res-irr}}.$ Applying Proposition \ref{prop:rk-sing} yields \linebreak $\codim_X Y\cap \pi^{-1}(\mathbf{x}) > l+3.$ Taking the union over all $\mathbf{x} \in A^\ell,$ we conclude that $\codim_X Y > 3,$ as desired.
\end{proof}

\begin{lemma}\label{good-res}
Let $\ul{f},\ul{g}$ be as above and assume that $e_{i}'>0$ for all $i\in[s].$ Suppose that $\ell\ge \ell_0(\ul{e},t)$ and $\str(\ul{f},\ul{g}) \ge C_{\ref*{good-res}}(\ul{d},\ul{e},\ell,t)$ and let $h\not\in I(Z)$. Then for generic $(v,\mathbf{w})\in V(\res_{\ul{f}})$ the following conditions hold:
\begin{enumerate}
\item $\str \big( \res_{\ul{g}}(v,\mathbf{w}) \big) \ge t$ and
\item $\res_{h}(v,\mathbf{w}) $ does not vanish identically on $V(\res_{\ul{g}}(v,\mathbf{w})) \subset \A^\ell$.
\end{enumerate}
\end{lemma}

\begin{proof}
    First note that $V(\res_{\ul{f}})$ is irreducible by Lemma~\ref{res-irr} (applied with no $g_i$'s), so it is enough to prove that each of these conditions separately holds on a nonempty open subset of $V(\res_{\ul{f}}).$ By Proposition~\ref{prop:str-gen}, a generic choice of $\ul{h} = (h_1,\ldots,h_s) \in \prod_{i=1}^s \cP_{e_{i}'}(K^\ell)$ has $\str(\ul{h}) \ge t$. Combining Propositions~\ref{prop:surjective} and~\ref{prop:res-strength}, we see that the map $\res_{(\ul{f},\ul{g})}$ is surjective and so $\res_{\ul{g}} \restriction_{V(\res_{\ul{f}})}$ is also surjective. This implies that the first condition holds on a nonempty open subset of $V(\res_{\ul{f}}),$ as desired. Now we turn to the second condition. Let $X$ be the variety from the previous lemma, $H(v,\mathbf{w},\mathbf{x}) = \res_h(v,\mathbf{w})(\mathbf{x})$ and $\pi:X\to V(\res_{\ul{f}})$ the projection $\pi(v,\mathbf{w},\mathbf{a})=  (v,\mathbf{w}).$
    Note that: 
    \begin{enumerate}
        \item $X$ is irreducible by Lemma~\ref{res-irr},
        \item $H\not\in I(X)$ and
        \item $\pi$ is surjective.
    \end{enumerate}
    These three facts imply that for a generic choice 
    of $(v,\mathbf{w})\in V(\res_{\ul{f}})$ we have
    \[
    H(v,\mathbf{w},\cdot) = \res_h(v,\mathbf{w})\not\in I(\pi^{-1}(v,\mathbf{w})) = I(V(\res_{\ul{g}}(v,\mathbf{w}))). \qedhere
    \]
\end{proof}

Now we are ready to prove the proposition.

\begin{proof}[Proof of Proposition \ref{prop:multi}]
    Given $h\not\in I(Z),$ we want a $K$-point $x\in Z$ with $h(x)\neq 0.$ The proof is by induction on $r'(\ul{f}) = |\{i\in[r]:n_i > 0\}|,$ the base case $r' = 0$ being trivial. For the inductive step, assume without loss of generality that $n_r>0.$ Let $C = C(d,n_r)$ be the constant from hypothesis $\Sigma^{\dag}(d).$ Apply the previous lemma with
    \begin{enumerate}
        \item $V=V_1\times\ldots\times V_{r-1},\ W = V_r,$
        \item $f_i$ given by $(f_{i,j})_{i\in [r-1],j\in [n_i]},$  
        \item and $g_j$ given by $f_{r,j}.$ 
    \end{enumerate}
      Assuming $l$ is sufficiently large, for generic $(v,\mathbf{w})\in V(\res_{\ul{f'}})$ we have $\str(\res_{f_{r,j}}((v,\mathbf{w})) \ge C$ and $\res_h(v,\mathbf{w}) \not\in I(V(\res_{f_{r,j}}(v,\mathbf{w}))).$  By the inductive hypothesis for $r'-1$, we can find a $K$-point $(v,\mathbf{w})\in V(\res_{\ul{f'}})$ satisfying these conditions. Hypothesis $\Sigma^{\dag}(d)$ then yields a $K$-point $\mathbf{x} \in V(\res_{f_{r,j}}(v,\mathbf{w}))$ with $\res_{h}(v,\mathbf{w})(\mathbf{x}) \neq 0.$ The $K$-point $(v,\mathbf{x}\cdot \mathbf{w})$ is therefore contained in $Z \cap (h\neq 0),$ as desired.
\end{proof}

\section{Diagonalization} \label{s:diag}

In \S \ref{s:diag}, we carry out Birch's ``all at once'' approach to finding orthogonal vectors, which allows us to reduce Theorem~\ref{main} to the case of diagonal forms.

\subsection{Orthogonal decompositions}

We recall the following important idea from \cite{BDS}:

\begin{definition}
Let $f \in \cP_d(V)$. Subspaces $V_1, \ldots, V_r$ of $V$ are \emph{$f$-orthogonal} if
\begin{displaymath}
f(v_1+\cdots+v_r) = f(v_1)+\cdots+f(v_r)
\end{displaymath}
for all $v_i \in V_i$. Vectors $v_1,\ldots,v_r\in V$ are $f$-\emph{orthogonal} if they span $f$-orthogonal lines.
\end{definition}

Concretely, if $x_{i,j}$, for $1 \le j \le n_i$, are coordinates on $V_i$, then the $V_i$'s are $f$-orthogonal if and only if the restriction of $f$ to $V_1 \oplus \cdots \oplus V_r$ has no mixed terms, i.e., each monomial that appears is a product of variables $x_{i,j}$ with $i$ fixed and $j$ varying.

We now use the $\Sigma^{\dag}(d)$ hypothesis to construct orthogonal spaces with some additional properties. This is the main argument where we apply Birch's ``all at once'' approach, and is our key technical improvement over \cite{BDS}.

\begin{proposition} \label{prop:ortho}
Suppose $\Sigma^{\dag}(d)$ holds. Let $f_1, \ldots, f_r$ be forms of odd degrees $\le d$ on a vector space $V$, let $g$ be a polynomial that does not vanish identically on $Z$, and let $s, n \in \N$ be given. Assume that the strength of $f_1, \ldots, f_r$ is at least $C_{\ref*{prop:ortho}}(r,d,n,s)$. Then we can find linearly independent subspaces $V_1, \ldots, V_{n+1}$ of $V$ such that:
\begin{enumerate}
\item The spaces $V_1, \ldots, V_{n+1}$ are $f_i$-orthogonal for all $i$.
\item The restriction of $f_1, \ldots, f_r$ to $V_i$ has strength $\ge s$ for each $1 \le i \le n$.
\item The polynomial $g$ does not vanish identically on $Z \cap V_{n+1}$.
\end{enumerate}
\end{proposition}

\begin{proof}
Let $\ell$ be a large integer and let $\cV=V^{\ell (n+1)}$; we denote points of $\cV$ by $(v_{1,\bullet}, \ldots, v_{n+1,\bullet})$, where each $v_{i,\bullet}$ is an $\ell$-tuple of vectors. We regard $\cV$ as the parameter space for the object we seek: given a point in $\cV$, we obtain spaces $V_1, \ldots, V_{n+1}$ by taking $V_i$ to be the span of the $\ell$-tuple $v_{i,\bullet}$. We let $X$ be the closed subvariety of $\cV$ where (a) holds. We show that $X$ is irreducible and that $X(K)$ is Zariski dense in $X$. We also show that there are non-empty open sets $U_1$, $U_2$, and $U_3$ of $X$ such that (a) holds on $U_1$, (b) holds on $U_2$, and linear independence holds on $U_3$. To complete the proof, we simply take a $K$-point of $U_1 \cap U_2 \cap U_3$.

\textit{The space $X$.} Let $W_i=K^{\ell}$ and $W=W_1 \times \cdots \times W_{n+1}$. Let $d_i$ be the degree of $f_i$. Let
\begin{displaymath}
\theta \colon \cV \to \prod_{i=1}^r \cP_{d_i}(W)
\end{displaymath}
be the restriction map $\res_{\ul{f}}$. This map has high strength by Proposition~\ref{prop:res-strength0}. Now, we have a decomposition
\begin{displaymath}
\cP_d(W) = \prod \cP_e(W),
\end{displaymath}
where the product is over multi-degrees $e=(e_1, \ldots, e_{n+1})$ of total degree $d$, and $\cP_e(W)$ is the space of multi-homogeneous functions on $W$ of multi-degree $e$. Let $\cP^1_d(W)$ be the product of those $\cP_e(W)$ where $e$ has one entry equal to $d$ and other entries equal to~0, and let $\cP^2_d(W)$ be the product of the remaining factors. Note that $\cP^1_d(W)$ consists of forms $h$ for which $W_1, \ldots, W_{n+1}$ are $h$-orthogonal, and we have an isomorphism $\cP^1_d(W) \cong \prod_{i=1}^{n+1} \cP_d(W_i)$.

Let $\theta^1$ and $\theta^2$ be the components of $\theta$ corresponding to $\cP^1$ and $\cP^2$. Thus $X=\theta_2^{-1}(0)$. Since $\theta_2$ is high strength, it follows that $X$ is geometrically irreducible. Moreover, Proposition~\ref{prop:multi} shows that $X(K)$ is dense in $X$; note that the equations defining $X$ are multi-homogeneous and each have odd degree $<d$ in one of the $\ell (n+1)$ components. Since $\theta$ is surjective (as a map of varieties), it follows that
\begin{displaymath}
\theta_1 \colon X \to \prod_{i=1}^r \cP^1_{d_i}(W)
\end{displaymath}
is surjective (as a map of varieties).

\textit{The first open set.} For each $1 \le i \le n$, let $U'_{1,i}$ be a non-empty open subset of $\prod_{j=1}^r \cP_{d_j}(W_i)$ consisting of tuples of strength at least $s$; such a set exists by Proposition~\ref{prop:str-gen}, provided $\ell$ is sufficiently large. Put $U'_{1,n+1}=\prod_{j=1}^r \cP_{d_j}(W_{n+1})$, and let $U'_1=\prod_{i=1}^{n+1} U'_{1,i}$, which is an open subset of $\prod_{i=1}^r \cP^1_{d_i}(W)$. Finally, we take $U_1=\theta_1^{-1}(U'_1)$. This is a non-empty subset of $X$ since $\theta_1$ is surjective on $X$.

\textit{The second open set.} Consider the map
\begin{displaymath}
\pi \colon \cV \to V^{\ell}, \qquad (v_{1,\bullet}, \ldots, v_{n+1,\bullet}) \mapsto v_{n+1,\bullet}.
\end{displaymath}
The map $\pi \vert_X$ is surjective: indeed, given any $v_{n+1,\bullet} \in V^{\ell}$, we obtain a point $(v_{1,\bullet}, \ldots, v_{n+1,\bullet})$ in $X$ by putting $v_{i,\bullet}=0$ for $1 \le i \le n$. Let $U_2'$ be a non-empty open subset of $V^{\ell}$ such that $g$ does not vanish identically on $Z \cap \operatorname{span}(v_{\bullet})$ for $v_{\bullet} \in V^{\ell}$; this exists by the following lemma. Finally, take $U_2=\pi^{-1}(U_2') \cap X$. This is non-empty since $\pi \vert_X$ is surjective.

\textit{The third open set.} Let $\tilde{U}_3$ be the open subset of $\cV$ consisting of points $(v_{i,j})$ such that the $(n+1)\ell$ vectors $v_{i,j}$ are linearly independent. Put $U_3=\tilde{U}_3 \cap X$. To show that $U_3$ is non-empty, we construct a $\ol{K}$-point of $U_3$. To this end, choose a linear subspace $W$ of $\ol{K} \otimes V$ of dimension $(n+1)\ell$ such that each $f_i$ restricts to~0 on $W$; the existence of such a subspace follows from Brauer's theorem, for instance. We take the $v_{i,j}$ to be any basis of $W$, and this will be a $\ol{K}$-point of $U_3$.
\end{proof}

\begin{lemma}
Let $Z$ be a closed subvariety of $V$ of codimension $\le r$ defined by homogeneous equations, let $g \in \cP(V)$ be a polynomial that does not vanish identically on $Z$, and let $\ell>r$. Then there exists a non-empty Zariski open subset $U$ of $V^{\ell}$ such that $g$ does not vanish identically on $Z \cap \operatorname{span}(v_{\bullet})$ for $v_{\bullet} \in U$.
\end{lemma}

\begin{proof}
Replacing $Z$ with one of its irreducible components if necessary, we may assume that $Z$ is irreducible. Since $g$ does not vanish on $Z$, it follows that some homogeneous component $g'$ of $g$ does not vanish on $Z$. We let $\ol{Z} \subset \bP(V)$ be the projective variety associated to $Z$, and let $A \subset \ol{Z}$ be the open set defined by $g' \ne 0$.

Let $\Gr_{\ell}(V)$ be the Grassmannian of $\ell$-dimensional subspaces of $V$ and let $\cE \subset \Gr_{\ell}(V) \times V$ be the tautological bundle. Let $p \colon \bP(\cE) \to \bP(V)$ and $q \colon \bP(\cE) \to \Gr_{\ell}(V)$ be the projection maps. Let $\tilde{Z}=p^{-1}(\ol{Z})$, which is irreducible since $\ol{Z}$ is irreducible and $p$ is smooth with connected fibers. If $E$ is any $\ell$-plane in $V$ then $Z \cap E$ contains a non-zero point, since $\ell>r$, and so $q(\tilde{Z})=\Gr_{\ell}(V)$.

Let $\tilde{A}=p^{-1}(A)$, and let $B=q(\tilde{A})$. Since $\tilde{Z}$ is irreducible, it follows that $\tilde{A}$ is a dense open set in $\tilde{Z}$. Since $q(\tilde{Z})=\Gr_{\ell}(V)$, it follows that $B$ is a dense constructible subset of $\Gr_{\ell}(V)$. Note that $B$ consists of those $\ell$-planes $E$ such that $g'$ is not identically zero on $E \cap Z$; for such $E$, it follows that $g$ is also not identically zero on $E \cap Z$.

Since $B$ is a dense constructible set, it contains a non-empty open subset $B_0$ of $\Gr_{\ell}(V)$. We take $U$ to be the set of all linear independent tuples $v_{\bullet} \in V^{\ell}$ such that $\operatorname{span}(v_{\bullet}) \in B_0$.
\end{proof}

\subsection{Reduction to diagonal forms}

We now improve upon the orthogonal decomposition found above by leveraging the stronger hypothesis $\Sigma^*(\ul{d})$. We begin with a lemma. In what follows, $\ul{d}=(d_1, \ldots, d_r)$ is a tuple of odd integers.

\begin{lemma} \label{lem:diag}
Suppose $\Sigma^*(\ul{d})$ holds. Let $f_1, \ldots, f_r$ be forms of degrees $d_1, \ldots, d_r$ on $V$. Assume that the strength of $f_1, \ldots, f_r$ is at least $C_{\ref*{lem:diag}}(\ul{d})$. Then there exists a vector $v \in V$ such that $f_1(v) \ne 0$ and $f_i(v)=0$ for $2 \le i \le r$.
\end{lemma}

\begin{proof}
Let $X$ be defined by $f_2=\cdots=f_r=0$. The set $X(K)$ is Zariski dense in $X$ by $\Sigma^*(\ul{d})$, since the tuple $(d_2, \ldots, d_r)$ is strictly less than $\ul{d}$. The strength condition ensures that $f_1$ does not vanish identically on $X$ (Corollary~\ref{prime}), and so there is a $K$-point of $X$ at which $f_1$ does not vanish.
\end{proof}

The following is the main result we are after.

\begin{proposition} \label{prop:diag}
Suppose $\Sigma^*(\ul{d})$ holds. Let $f_1, \ldots, f_r$ be forms of degrees $d_1, \ldots, d_r$ on $V$, let $g$ be a polynomial that does not vanish identically on $Z$, and let $\ell \ge 0$ be given. Assume that the strength of $f_1, \ldots, f_r$ is at least $C_{\ref*{prop:diag}}(\ul{d}, \ell)$. Then we can find linearly independent subspaces $V_1, \ldots, V_r, W$ of $V$ such that:
\begin{enumerate}
\item The spaces $V_1, \ldots, V_r, W$ are $f_i$-orthogonal, and $f_i$ vanishes identically on $V_j$ for $j \ne i$.
\item $V_i$ has an $f_i$-orthogonal basis $v_{i,1}, \ldots, v_{i,\ell}$ such that $f_i(v_{i,j}) \ne 0$ for each $1 \le j \le \ell$.
\item $g$ does not vanish identically on $Z \cap W$.
\end{enumerate}
\end{proposition}

\begin{proof}
Let $d_0=\max(\ul{d})$. Since $\Sigma^*(\ul{d})$ holds, it follows that $\Sigma^{\dag}(d_0)$ holds as well. Apply Proposition~\ref{prop:ortho} to construct linearly independent subspaces $V'_{1,1}, \ldots, V'_{r, \ell}, W$ such that:
\begin{itemize}
\item The spaces $V'_{1,1}, \ldots, V'_{r, \ell}, W$ are $f_i$-orthogonal for all $i$.
\item The restriction of $f_1, \ldots, f_r$ to $V'_{i,j}$ has strength $\ge C_{\ref*{lem:diag}}(\ul{d})$ for each $i$, $j$.
\item The polynomial $g$ does not vanish identically on $Z \cap W$.
\end{itemize}
Applying Lemma~\ref{lem:diag}, pick $v_{i,j} \in V'_{i,j}$ such that $f_i(v_{i,j}) \ne 0$ but $f_k(v_{i,j})=0$ for $k \ne i$. Now take $V_i$ to be the span of $v_{i,1}, \ldots, v_{i,\ell}$.
\end{proof}

\section{Finishing the proof} \label{s:proof-end}

\subsection{Specializing diagonal forms}

We have now essentially reduced the proof of the main theorem to the case of diagonal forms. To handle these, we will specialize them to particularly simple forms. The following is the key result:

\begin{proposition} \label{prop:specialize}
Suppose $\Sigma^{\dag}(d)$ holds. Let $f=\sum_{i=1}^n c_i x_i^d$ be a diagonal form of odd degree $d$ on $V=K^n$ with each $c_i \in K$ non-zero. Then, for $n \ge C_{\ref*{prop:specialize}}(d)$, there exist linearly independent $v,w \in V$ such that $f(xv+yw)=xy^{d-1}+ay^d$ for some $a \in K$.
\end{proposition}

\begin{proof}
Let $m=N_K(d)$ be the constant from Definition~\ref{defn:birch}. Setting at most $m$ of the $x_i$'s to~0 if necessary, we assume $n=rm$. It will now be convenient to index coordinates on $V$ by $[r] \times [m]$ rather than $[n]$. We thus have $f=\sum_{i=1}^r \sum_{j=1}^m c_{i,j} x_{i,j}^d$. We let $e_{i,j}$, for $i \in [r]$ and $j \in [m]$, be a basis for $V$. By definition of $m$, for each $1 \le i \le r$, we can find a non-zero vector $u_{i,\bullet} \in K^m$ such that
\begin{displaymath}
c_{i,1} u_{i,1}^d + \cdots + c_{i,m} u_{i,m}^d = 0
\end{displaymath}
Applying a permutation, if necessary, we assume $u_{i,m} \ne 0$. For $\alpha \in K^r$, put
\begin{displaymath}
v = \sum_{i=1}^r \sum_{j=1}^m \alpha_i u_{i,j} e_{i,j}.
\end{displaymath}
Note that $f(v)=0$ for all $\alpha$. For $\beta \in K^r$, put $w = \sum_{i=1}^r \beta_i e_{i,m}$. We have
\begin{displaymath}
f(xv+yw) = \sum_{j=1}^d f^j(\alpha, \beta) x^{d-j} y^j
\end{displaymath}
where
\begin{displaymath}
f^j(\alpha, \beta) = \binom{d}{j} \sum_{i=1}^r c_{i,m} u_{i,m}^{d-j} \alpha_i^{d-j} \beta_i^j.
\end{displaymath}
We regard $f^j$ as a bihomogeneous polynomial on $K^r \times K^r$.

The form $f^j$ has strength at least $(r-1)/2$ by Proposition~\ref{prop:rk-sing}. Since the $f^j$ have different bidegrees, the sequence $f^1, \ldots, f^{d-1}$ has collective strength at least $\frac{r-1}{2(d-1)}$. Indeed, suppose some linear combination of the $f^j$'s were equal to $\sum_{i=1}^s g_i h_i$. Now take the bidegree $(d-j, j)$ piece of this equation. The bidegree $(d-j, j)$ piece of $g_i h_i$ is easily seen to have strength $\le d-1$, and so this would imply that some $f^j$ has strength $\le (d-1)s$.

We thus see that if $r$ is sufficiently large, then the $f^j$ form a prime sequence (Corollary~\ref{prime}). In particular, if $Z$ is the variety defined by the vanishing of $f^1, \ldots, f^{d-2}$, then $f^{d-1}$ does not vanish identically on $Z$. By Proposition~\ref{prop:multi}, $Z(K)$ is dense in $Z$ if $r$ is sufficiently large. We can thus find $\alpha, \beta \in K^r$ such that $f^j(\alpha, \beta)=0$ for $1 \le j \le d-2$ and $f^{d-1}(\alpha, \beta) \ne 0$. Rescaling, we can assume $f^{d-1}(\alpha, \beta)=1$. We thus have
\begin{displaymath}
f(xv+yw) = xy^{d-1}+ay^d,
\end{displaymath}
where $a=f^d(\alpha, \beta)$. If $d>1$ then $v$ and $w$ are necessarily linearly independent; if $d=1$ it is obvious that we can choose them so.
\end{proof}

\begin{corollary} \label{cor:specialize}
Suppose $\Sigma^{\dag}(d)$ holds, and let $f$ be as in the proposition. If $n \ge C_{\ref*{cor:specialize}}(d)$ then there exists linearly independent $u,v,w \in V$ such that
\begin{displaymath}
f(xv+yw+zu)=xy^{d-1}+ay^d+bz^d
\end{displaymath}
for scalars $a$ and $b$, with $b$ non-zero.
\end{corollary}

\begin{proof}
Simply take $u$ to be the final basis vector of $V$, and then find $v$ and $w$ in the span of the other basis vectors by using the proposition.
\end{proof}

\subsection{A normal form}

We can now use the results we have established so far to show that a collection of high strength forms can be specialized into a kind of normal form.

\begin{proposition} \label{prop:normal}
Suppose $\Sigma^*(\ul{d})$ holds. Let $f_1, \ldots, f_r$ be forms of degrees $d_1, \ldots, d_r$ on $V$ and let $g$ be a polynomial that does not vanish identically on $Z$. Assume that the strength of $f_1, \ldots, f_r$ is at least $C_{\ref*{prop:normal}}(\ul{d})$. We can then find three dimensional subspaces $V_1, \ldots, V_r$ of $V$ and another subspace $W$ of $V$, all linearly independent, such $g$ is not identically zero on $Z \cap W$ and the restriction of $f_i$ to $V_1 \oplus \cdots \oplus V_r \oplus W$ has the form
\begin{equation} \label{eq:normal}
x_i y_i^{d_i-1} + a_i y_i^{d_i} + b_i z_i^{d_i} + h_i(w_1, \ldots, w_n),
\end{equation}
where $x_i$, $y_i$, and $z_i$ are coordinates on $V_i$, the $w_i$'s are coordinates on $W$, $h_i \in \cP_{d_i}(W)$, and $a_i$ and $b_i$ belong to $K$, with $b_i$ non-zero.
\end{proposition}

\begin{proof}
This follows from Propositions~\ref{prop:diag} and Corollary~\ref{cor:specialize}.
\end{proof}

\subsection{Final steps}

We are finally ready to prove our main result.

\begin{proof}[Proof of Theorem~\ref{main}]
We suppose $\Sigma^*(\ul{d})$ holds and prove $\Sigma(\ul{d})$. Thus let $f_1, \ldots, f_r$ be given of degrees $d_1, \ldots, d_r$. Let $g$ be a function that does not vanish identically on $Z$. We must construct a $K$-point of $Z$ at which $g$ is non-vanishing. Apply Proposition~\ref{prop:normal}, and let $V'=V_1 \oplus \cdots \oplus V_r \oplus W$. The variety $Z \cap V'$ is defined by the vanishing of the forms \eqref{eq:normal}, and is easily seen to be irreducible (see \cite[Proposition~4.4]{BDS}). Moreover, $Z \cap V'$ is a rational variety, since once can solve for $x_i$ in each equation. It follows that the $K$-points of $Z \cap V'$ are Zariski dense. Since $g$ is non-vanishing on $Z \cap V'$, there is thus a $K$-point at which it is non-vanishing.
\end{proof}

\begin{remark}
In the above proof, it is important that $Z \cap V'$ is irreducible. This would not necessarily be the case if the $b_iz_i^{d_i}$ terms were absent, and this is why we have taken measures to ensure their presence.
\end{remark}

\section{Examples of Birch fields} \label{s:birch}

Any Brauer field is trivially a Birch field. We have also seen that any number field is a Birch field. This shows that there are Birch fields that are not Brauer fields, e.g., the field of rational numbers. We now give some additional examples.

\begin{proposition} \label{prop:real-birch}
The field $K=\R(t_1, \ldots, t_p)$ is Birch, for any $p \ge 1$.
\end{proposition}

\begin{proof}
Fix $d \ge 1$ odd and put $N=d^p+1$. We will show that $N_K(d) \le N$.

Consider an equation $f(x_1, \dots, x_n)=0$, where $f$ is a homogeneous form of odd degree $d$ with coefficients in $K$ and $n \ge N$. We show that this equation has a non-trivial solution. Multiplying by a non-zero element of $K$, we assume that the coefficients of $f$ belong to $\R[t_1, \ldots, t_p]$. Let $r$ be the maximum total degree of a coefficient of $f$, let $s$ be a number to be determined later, and write $x_i=\sum_{\vert a \vert \le s} y_{i,a} t^a$, where the sum is over multi-degrees $a=(a_1, \ldots, a_p)$ of total degree at most $s$ and $y_{i,a} \in \R$. Then $f(x)=\sum_{\vert a \vert \le r+ds} f_a(y) t^a$, where $f_a$ is a homogeneous form of degree $d$ in the $y$ variables. A non-trivial solution to the system of equations $f_a(y)=0$ (in $\R$) will yield a non-trivial solution to $f(x)=0$ (in $K$).

Now, the number of equations (i.e., the number of $f_a$'s) is the number of monomials in $t_1, \ldots, t_p$ of total degree at most $r+ds$, which is asymptotically $(r+ds)^p/p!$ (for $r$, $d$, $p$ fixed and $s$ varying). The number of variables (i.e., the number of $y$'s) is asymptotically $n s^p/p!$ (in the same regime). Since $n>d^p$, the number of variables grows faster than the number of variables with $s$. Thus, given $f$, we can choose $s$ so that there are more variables than equations, and the existence of a non-trivial solution now follows from Lemma~\ref{lem:real-soln}.
\end{proof}

\begin{lemma} \label{lem:real-soln}
Let $f_1, \ldots, f_r$ be odd degree homogeneous forms over $\R$ in $n>r$ variables. Then the system of equations $f_i=0$ has a non-trivial solution in $\R$.
\end{lemma}

\begin{proof}
In fact, the real dimension of the solution set is at least $n-r$, see, e.g., \cite[p.~247]{Schmidt}.
\end{proof}

We note that $\R(t_1, \ldots, t_p)$ is formally real (i.e., $-1$ is not a sum of squares), and so it is not a Brauer field \cite[\S 2.1]{BDS}. We also have the following general procedure to generate new Birch fields:

\begin{proposition} \label{prop:Birch-finite}
Let $K$ be a Birch field and let $L/K$ be a finite extension field. Then $L$ is a Birch field.
\end{proposition}

\begin{proof}
The same argument from \cite[Proposition~2.7]{BDS} applies. For the sake of completeness, we recall it. Let $\alpha_1, \ldots, \alpha_m$ be a $K$-basis of $L$, and let $f$ be a diagonal form over $L$ of odd degree $d$ in variables $x_1, \ldots, x_n$. Write $x_i=\sum_{j=1}^m \alpha_j y_{i,j}$, where the variables $y_{i,j}$ take values in $K$. By similarly expressing the coefficients of $f$ in the $\alpha$ basis, we obtain $f(x)=\sum_{j=1}^m \alpha_j f_j(y)$, where each $f_j$ is a homogeneous form over $K$ of degree $d$ in the $nm$ variables $y_{i,j}$. If $nm$ exceeds the constant $C_{\ref{Birch}}(d,m)$ then Birch's theorem (Theorem~\ref{Birch}) shows that there is a non-zero $y \in K^{nm}$ such that $f_j(y)=0$ for all $j$. This gives a non-zero $x \in L^n$ such that $f(y)=0$. We thus see that $N_L(d)$ is bounded above by the ceiling of $m^{-1} \cdot C_{\ref{Birch}}(d,m)$.
\end{proof}

\begin{remark}
The above argument, combined with the theorem of Leep--Star, shows that the class of Leep--Star fields is closed under finite extensions.
\end{remark}

Combining the above two propositions, we obtain the following corollary.

\begin{corollary} \label{cor:fg-birch}
Any finitely generated extension of $\R$ is a Birch field.
\end{corollary}

\begin{remark}
The above results hold with $\R$ replaced by any real-closed field, such as the field of real algebraic numbers.
\end{remark}

\begin{remark}
Recall that a $C_p$ field, in the sense of Lang \cite{Lang}, is one for which any homogeneous form of degree $d \ge 1$ in $>d^p$ variables admits a non-trivial solution. Any $C_p$ field is trivially a Brauer field \cite[\S 2.4]{BDS}. If $K$ is a $C_p$ field then $K(t)$ is a $C_{p+1}$ field; this result was first proved by Lang, but the argument goes back to Tsen. It is exactly this argument we have used in the proof of Proposition~\ref{prop:real-birch}.

This suggests a possible ``Birch generalization'' of the theory of $C_p$ fields: define a $C^+_p$ field to be one for which any homogeneous form of \emph{odd} degree $d \ge 1$ in $>d^p$ variables admits a non-trivial solution. Any $C^+_p$ field is trivially a Birch field. The proof Proposition~\ref{prop:real-birch} actually shows that $\R(t_1, \ldots, t_p)$ is a $C^+_p$ field. Is it a general fact that that if $K$ is a $C^+_p$ field then $K(t)$ is a $C^+_{p+1}$ field? Even more generally, one can attempt to find a ``Leep--Star generalization'' of the $C_p$ theory.
\end{remark}

\begin{question}
We close with a few questions about Birch fields:
\begin{enumerate}
\item Suppose $K$ is a field and there is a quadratic extension $L/K$ with $L$ a Brauer field. Is $K$ a Birch field?
\item Does every Birch field admit a finite extension that is a Brauer field?
\item Is $\R(\!(t_1, \ldots, t_p)\!)$ a Birch field?
\item Is $\Q(t)$ a Birch field? Note that if $K$ is an imaginary quadratic number field then it is known that $K$ is a Brauer field, but unknown if $K(t)$ is a Brauer field. \qedhere
\end{enumerate}
\end{question}


\begin{thebibliography}{BDLZ}

\bibitem[AH]{AH} Tigran Ananyan, Melvin Hochster. Small subalgebras of polynomial rings and Stillman’s conjecture. \textit{J.\ Amer.\ Math.\ Soc.} \textbf{33} (2020), no.~1, pp.~291--309. \DOI{10.1090/jams/932} \arxiv{1610.09268}

\bibitem[Bir1]{Birch} B. J. Birch. Homogeneous forms of odd degree in a large number of variables. \textit{Mathematika} \textbf{4} (1957), no.~2, pp.~102--105. \DOI{10.1112/S0025579300001145}

\bibitem[Bir2]{Birch2} B. J. Birch. Forms in Many Variables. \textit{Proceedings of the Royal Society of London}. Series A, Mathematical and Physical Sciences, vol. 265, no. 1321, 1962, pp. 245–63.

\bibitem[Bra]{Brauer} Richard Brauer. A note on systems of homogeneous algebraic equations. \textit{Bull.\ Amer.\ Math.\ Soc.} \textbf{51} (1945), no.~10, pp.~749--755. \DOI{10.1090/S0002-9904-1945-08440-7}

\bibitem[BDLZ]{BDLZ} Arthur Bik, Jan Draisma, Amichai Lampert, Tamar Ziegler. Strength and partition rank under limits and field extensions. In preparation.

\bibitem[BDS1]{BDS} Arthur Bik, Jan Draisma, Andrew Snowden. Two improvements in Brauer's theorem on forms. \arxiv{2401.02067}

\bibitem[BDS2]{BDS2} Arthur Bik, Jan Draisma, Andrew Snowden. The geometry of polynomial representations in positive characteristic. \arxiv{2406.07415}

\bibitem[ESS]{ESS} Daniel Erman, Steven V Sam, Andrew Snowden. Cubics in 10 variables vs. cubics in 1000 variables: Uniformity phenomena for bounded degree polynomials. \textit{Bull.\ Amer.\ Math.\ Soc.} \textbf{56} (2019), pp.~87--114. \DOI{10.1090/bull/1652} \arxiv{1809.09402}

\bibitem[FM]{FM} Christopher Frei, Manfred Madritsch. Forms of differing degrees over number fields. \textit{Mathematika} \textbf{63} (2017), no.~1, pp.~92--123. \DOI{10.1112/S0025579316000206} \arxiv{1412.6419}

\bibitem[Gre]{Greenberg} Marvin J. Greenberg. Lectures on forms in many variables. W. A. Benjamin, Inc., New
York-Amsterdam, 1969.

\bibitem[GT]{GT} Ben Green, Terence Tao. The distribution of polynomials over finite fields, with applications to the Gowers norms. \textit{Contrib.\ Discrete Math.} \textbf{4} (2009), no.~2, pp.~1--36. \DOI{10.11575/CDM.V4I2.62086} \arxiv{0711.3191}

\bibitem[KLP]{KLP} David Kazhdan, Amichai Lampert, Alexander Polishchuk. Schmidt rank and singularities. \textit{Ukrainian Math.\ J.} \textbf{75} (2024), no.~9, pp.~1420--1442. \DOI{10.1007/s11253-024-02270-6} \arxiv{2104.10198}

\bibitem[KZ1]{KZ-properties} David Kazhdan, Tamar Ziegler. Properties of high rank subvarieties of affine spaces. \textit{Geom. Funct. Anal.} \textbf{30} (2020) no.~4 pp.~1063--1096 \DOI{10.1007/s00039-020-00542-4} \arxiv{1902.00767}

\bibitem[KZ2]{KZ-applications} David Kazhdan, Tamar Ziegler. Applications of algebraic combinatorics to algebraic geometry. \textit{Indag.\ Math.\ (N.S.)} \textbf{32} (2021), no.~6, pp.~1412--1428. \DOI{10.1016/j.indag.2021.09.002} \arxiv{2005.12542}

\bibitem[Lam]{Lam} Amichai Lampert. Small ideals in polynomial rings and applications. \arxiv{2309.16847}

\bibitem[Lan]{Lang} Serge Lang. On quasi-algebraic closure. \textit{Ann.\ Math.} \textbf{55} (1952), no.~2, pp.~373--390. \DOI{10.2307/1969785}

\bibitem[LS]{LS} David B. Leep, Colin L. Starr. A generalization of a theorem of Birch. \textit{Comm.\ Alg.} \textbf{37} (2009), no.~8, pp.~2640--2648. \DOI{10.1080/00927870902747381}

\bibitem[Pec]{Peck} L. G. Peck. Diophantine equations in algebraic number fields. \textit{Amer.\ J.\ Math.} \textbf{71} (1949), no.~2, pp.~387--402. \DOI{10.2307/2372253}

\bibitem[SGA]{SGA} A. Grothendieck. Cohomologie locale des faisceaux coh\'erents et th\'eor\`emes de {L}efschetz locaux et globaux (SGA2). \textit{Advanced Studies in Pure Mathematics} (1968) 

\bibitem[Sch]{Schmidt} Wolfgang M. Schmidt. The density of integer points on homogeneous varieties. \textit{Acta Math.} \textbf{154} (1985), no.~3--4, pp.~243--296. \DOI{10.1007/BF02392473}

\bibitem[Ski]{Skinner} C. M. Skinner. Forms over number fields and weak approximation. \textit{Compos.\ Math.} \textbf{106} (1997), no.~1, pp.~11--29. \DOI{10.1023/A:1000129818730}

\end{thebibliography}
\end{document}